\documentclass[12pt]{amsart}

\usepackage{verbatim,amsmath,latexsym,amssymb,amsbsy}
\usepackage[all]{xy}

\theoremstyle{plain}
\newtheorem{thm}{Theorem}[section]

\newtheorem{prop}[thm]{Proposition}
\newtheorem{cor}[thm]{Corollary}

\theoremstyle{definition}

\newtheorem{df}[thm]{Definition}

\newtheorem{ex}[thm]{Example}
\newtheorem{ex-notn}[thm]{Example/Notation}

\newtheorem{rem}[thm]{Remark}

\newcommand*{\calHom}{\mathcal{H}\mathit{om}}                  % `sheaf' Hom
\def\det{\operatorname{det}}

\def\dim{\operatorname{dim}}

\def\ext{\operatorname{Ext}}

\def\hom{\operatorname{Hom}}

\def\Pic{\operatorname{Pic}}

\def\CC{{\mathbb C}}     
\def\GG{{\mathbb G}}

\def\PP{{\mathbb P}}

\def\ZZ{{\mathbb Z}}

\def\calF{{\mathcal F}}

\def\calI{{\mathcal I}}

\def\calO{{\mathcal O}}

\def\ilim{{\lim\limits_{\longleftarrow}}}

\def\ra{\rightarrow}
\def\into{\hookrightarrow}
\def\onto{\to\hskip-1.7ex\to}
\def\Mtwo{{\em Macaulay} 2\expandafter}

\numberwithin{equation}{section}
\begin{document}

\title{Splitting criteria for vector bundles on higher dimensional varieties}
\author{Parsa Bakhtary}

\address{Department of Mathematics and Statistics \\
King Fahd University of Petroleum and Minerals \\
Dhahran, Saudi Arabia 31261}

\thanks{Partially supported by a KFUPM FT grant.}
\subjclass{14J60; 14F05}
\keywords{vector bundle, splitting, Horrocks' criterion}
\email{pbakhtary@kfupm.edu.sa}

\begin{abstract}
We generalize Horrocks' criterion for the splitting of vector bundles on projective space by establishing an analogous 
splitting criterion for vector bundles on a class of smooth complex projective varieties of dimension $\geq 4$, 
over which every extension of line bundles splits. 
\end{abstract}

\maketitle

\section{Introduction}
In algebraic geometry there is a rich history of studying when a vector bundle over a projective space splits, i.e. is isomorphic to a direct sum of line bundles.  Grothendieck first used cohomological methods in sheaf theory to prove his celebrated theorem which says that every vector bundle over $\PP^1$ splits as a direct sum of line bundles \cite{Gr2}.  This was followed by Horrocks' famous criterion, which announced that a vector bundle on $\PP^n$, $n \geq 3$, splits iff its restriction to a hyperplane $H=\PP^{n-1} \subset \PP^n$ splits \cite{Horr}.
\\

Soon after came the notoriously difficult conjectures of Hartshorne \cite{H3}, which state that all vector bundles of rank $2$ on $\PP^n$ with $n \geq 7$ must split, though over $\CC$ no non-splitting (indecomposable) $2$-bundle over $\PP^5$ is known.  Over the complex numbers, the Horrocks-Mumford bundle is the only non-splitting $2$-bundle known on $\PP^4$ \cite{HM}, and its existence is far from obvious.  It should be mentioned that there are rank $2$ indecomposable bundles on $\PP^5$ in characteristic $2$ \cite{T} and on $\PP^4$ in any positive characteristic different from the Horrocks-Mumford bundle \cite{K}.  
\\

There has been a formidable body of work dedicated to finding splitting criteria and constructing indecomposable bundles over projective space, and the well-known but out of print book by Okonek, Schneider, and Spindler \cite{Ok} gives an excellent survey of progress made in this direction up until 1980.  There has also been much work since then, with many notable results \cite{A}, \cite{B2}, \cite{KRP1}, \cite{KRP2}, \cite{Mal}, \cite{R}, \cite{S}, \cite{Z}.  In addition to splitting criteria for $r$-bundles on multiprojective spaces \cite{BM2}, \cite{CMR}, cones over rational normal curves \cite{B1}, and blowings up of the plane \cite{BM1}, extensions of Horrocks' criterion to Grassmannians and quadrics have been established \cite{A}, \cite{O}. Furthermore the splitting of $2$-bundles on hypersurfaces in $\PP^4$ and $\PP^5$ has been studied \cite{Mad}, \cite{CM} and generalized in \cite{KRR1} and \cite{KRR2}, which show respectively that arithmetically Cohen-Macaulay $2$-bundles on general hypersurfaces of degree $\geq 6$ in $\PP^4$ and on general hypersurfaces of degree $\geq 3$ in $\PP^5$ must split.  Moreover, notions such as uniform vector bundles have been generalized to Fano manifolds \cite{W}.
\\

However, to the author's knowledge, the literature lacks a study of when a Horrocks' type criterion occurs on arbitrary smooth projective varieties.  The spirit in which we pursue this question is similar to that of Horrocks': when can we reduce the splitting of a vector bundle $E$ on a smooth projective variety $X$ to the splitting of the restriction $E_{|Y}$ for a suitable proper closed subscheme $Y \subset X\,$?  Horrocks showed that as soon as the dimension of a projective space is at least three, the splitting of a vector bundle on that projective space is equivalent to the splitting of its restriction to a hyperplane.  In this scenario the restriction map $\Pic(\PP^n) \overset{\sim} \ra \Pic(H)$ is an isomorphism, so the line bundles on the subscheme $H$ are precisely those coming from $\PP^n$, no more, no less.  Thus if $E_{|H}$ splits, we already have a suitable candidate on $\PP^n$ that $E$ ought to be isomorphic to, should it split.  The dimension of the hyperplane being at least $2$ is crucial, since any non-splitting bundle must split when restricted to a line $\PP^1$ by Grothedieck's theorem.
\\

We remedy this issue for higher dimensional varieties using the Grothendieck-Lefschetz theorem on Picard groups (see \cite{H2} for an exposition), which says that if $X$ is a smooth complex projective variety of dimension $n \geq 4$, then given any ample effective divisor $D$ (not necessarily reduced) on $X$, then the natural restriction map $\Pic(X) \ra \Pic(D)$ is an isomorphism.  In this way, we ensure that the line bundles on our divisor $D$ are precisely those coming from $X$, as in Horrocks' situation with projective space.  Then, assuming $E_{|D}$ splits over $D$, our task is to try to lift a given isomorphism $E_{|D} \overset{\sim} \ra \bigoplus {L_i}_{|D}$ to one on $X$, or to find the obstruction to such a lifting.
\\

Though this lifting does not exist in general, it can be found on a certain class of varieties.  We call a scheme $X$ a {\em{Horrocks scheme}} if $H^1(X,L)=H^2(X,L)=0$ for every line bundle $L$ on $X$.   A Horrocks scheme is like 
projective $n \geq 3$ space in the sense that every extension of line bundles splits.  Here a Horrocks' type criterion holds.  

Our main result is the following theorem.

\begin{thm}
Let $X$ be a smooth complex projective Horrocks variety of dimension $n \geq 4$.  A vector bundle $E$ on $X$ splits iff 
$E_{|D}$ splits over $D$, where $D$ is an ample effective divisor on $X$.
\end{thm}

\subsection{Acknowledgements}
The author is grateful to J. W{\l}odarczyk for interesting discussions and to J. Wi\'{s}niewski for helpful and informative emails.  The author would also like to give a heartfelt thanks to N. Mohan Kumar, who spotted an error in the first version of this paper and told the author of a class of examples of non-splitting bundles that restrict to split bundles.  The author also thanks the referee for valuable suggestions regarding the references and for sharpening the rank of the bundle in example 4.9.

\section{Preliminaries}
Nearly all of the results we use are familiar to a seasoned student of algebraic geometry, and can be found throughout \cite{H1}.  Throughout this paper we will work over $\CC$.  In this section we mention some of the deeper theorems that are relevant to the proofs in the next section.  We first state Horrocks' criterion in its full form. 

\begin{thm}[Horrocks]
Let $E$ be a rank $r$ vector bundle on $\PP^n$.  Then $E$ splits iff $H^i(\PP^n,E(k))=0$ for every $k \in \ZZ$ and every $i$ with $0<i<n$.
\end{thm}

\begin{proof}
See \cite{Horr} or \cite{Ok}.
\end{proof}

\begin{cor}[Horrocks]
Let $E$ be a rank $r$ vector bundle on $\PP^n$, with $n \geq 3$.  Then $E$ splits iff its restriction $E_{|H}$ to a hyperplane $H \cong \PP^{n-1} \subset \PP^n$ splits.
\end{cor}

Thus, by induction it suffices to find a plane $P \cong \PP^2 \subset \PP^n$ such that $E_{|P}$ splits.  

Recall the formal completion of $X$ along a closed subscheme $Z$ defined by the ideal $\calI \subset \calO_X$ is the ringed space $(\hat{X},\calO_{\hat{X}})$ whose topological space is $Z$ and whose structure sheaf is $\ilim (\calO_X / \calI^m)$.  Given a coherent sheaf $\calF$ on $X$ we define the completion of $\calF$ along $Z$, denoted $\hat{\calF}$ to be the sheaf $\ilim (\calF / \calI^m \calF)$ on $Z$, which has the natural structure of an $\calO_{\hat{X}}$-module.  

Our most important gadget is the Grothendieck-Lefschetz theorem on Picard groups.  We do not require the most general version.

\begin{thm}[Grothendieck]
Let $D$ be an ample effective (not necessarily reduced) divisor on a smooth complex projective variety $X$ of dimension $n \geq 4$.  Then the natural restriction map $\Pic(X) \ra \Pic(D)$ is an isomorphism.
\end{thm}

\begin{proof}
See \cite{Gr1} or \cite{H2}.
\end{proof}

In conjunction with II, Ex. 9.6 in \cite{H1}, we have the following chain of natural isomorphisms for any positive integer $m$:
\[ \Pic(X) \overset{\sim} \ra \Pic(\hat{X}) \overset{\sim} \ra \ilim \Pic(mD) \overset{\sim} \ra \Pic(mD) \overset{\sim} \ra \Pic(D) \] 
whose composition is the natural restriction map isomorphism mentioned in the theorem.

\section{Arbitrary Varieties}
We first study the splitting behavior of a vector bundle restricted to the formal completion of a projective manifold along an ample effective divisor, and show this is equivalent to the splitting of the bundle itself.

\begin{prop}
Let $X$ be a smooth complex projective variety of dimension $n \geq 4$, and let $E$ be a vector bundle of rank $r$ on $X$.  Then $E$ splits over $X$ iff $\hat{E}$ splits over $\hat{X}$, where $\hat{X}$ is the completion of $X$ along an ample effective divisor $D$.
\end{prop}

\begin{proof}
For the forward direction, suppose $E$ splits over $X$, i.e. $E \cong \bigoplus L_i$.  Then $\hat{E} \cong \bigoplus \hat{L_i}$, and each $\hat{L_i}$ is a line bundle on $\hat{X}$.

For the other direction, suppose that $\hat{E}$ splits as a direct sum of line bundles on $\hat{X}$.  Then since 
\[ \Pic(X) \cong \Pic(\hat{X}) \cong \ilim \Pic(mD) \cong \Pic(mD) \cong \Pic(D)\] 
we must have that $\hat{E} \cong \bigoplus \hat{L_i}$, for some line bundles $L_i$ on $X$.  Set $F:= \bigoplus L_i$.  Tensoring the short exact sequence 
\[ 0 \ra \calO_X(-mD) \ra \calO_X \ra \calO_{mD} \ra 0 \]
with $F^* \otimes E \cong \calHom(F,E)$ we obtain 
\[ 0 \ra \calO_X(-mD) \otimes F^* \otimes E \ra F^* \otimes E \ra F^*_{|mD} \otimes E_{|mD} \ra 0 \]
Choosing $m \gg 0$ and using Serre duality plus the fact that $\calO_X(D)$ is ample we can force $H^1(X,\calO_X(-mD) \otimes F^* \otimes E)=0$ and we get a surjection 
\[ \hom(F,E) \onto \hom(F_{|mD},E_{|mD}) \]
We can lift a given isomorphism $\varphi: F_{|mD} \overset{\sim} \ra E_{|mD}$ (this is just our original isomorphism $\hat{F} \overset{\sim} \ra \hat{E}$ restricted to a finite thickening $mD$) to a homomorphism $\psi: F \ra E$ on $X$.  The bundles $E$ and $F$ have the same rank and first Chern class, the latter because $ \calO_{mD} \cong \det E_{|mD} \otimes \det F^*_{|mD} \cong \calO_X(c_1(E)-c_1(F))_{|mD}$ implies that $ \calO_X(c_1(E)) \cong \calO_X(c_1(F))$ on $X$ since the restriction map $\Pic(X) \overset{\sim} \ra \Pic(mD)$ is an isomorphism.  Thus, 
\[ \det \psi \in \hom(\det F, \det E) \cong H^0(X,\calO_X(c_1(E)-c_1(F))) \cong H^0(X,\calO_X) \cong \CC \]
is a nonzero constant since $\psi$ restricts to an isomorphism on $mD$.  Hence $\psi$ is invertible.
\end{proof}

\begin{rem}
This proposition illustrates that, in the above setting, if $E_{|D}$ splits on a sufficiently positive divisor $D$ on $X$, then $E$ must split over $X$.  One possible approach is to make positivity assumptions on $D$ in terms of the Chern classes of $E$.
\end{rem}

\section{Horrocks Schemes}

We begin with the definition of a splitting scheme and a Horrocks scheme, which capture a cohomological feature of line bundles on projective spaces, and give some examples.

\begin{df}
A scheme $X$ is called a {\em{splitting scheme}} if $H^1(X,L)=0$ for any line bundle $L$ on $X$.  Equivalently, $\ext^1(L,M)=0$ for any line bundles $L$, $M$ on $X$, i.e. any extension of line bundles splits.
\end{df}

\begin{df}
A scheme $X$ is called a {\em{Horrocks scheme}} if $H^i(X,L)=0$ for $i=1,2$ and any line bundle $L$ on $X$.
\end{df}

\begin{rem}
Notice that for a smooth projective variety to be a splitting scheme, its dimension must be at least two.  For a curve we would have by Serre duality that $H^1(C,L^{-1}) \cong H^0(C,L \otimes \omega_C) \neq 0$ for a sufficiently ample line bundle $L$ on $C$.  Similarly a smooth projective variety must be of dimension at least three in order to be a Horrocks scheme, and for 
threefolds these notions are equivalent.
\end{rem}

The following are examples of Horrocks schemes.  It is clear how to adjust the constructions to obtain splitting schemes.

\begin{ex}
Clearly projective space $\PP^n$ with $n \geq 3$ is a Horrocks scheme.
\end{ex}

\begin{ex}
If $X$ and $Y$ are projective Horrocks varieties, then $\Pic(X \times Y) \cong \Pic(X) \times \Pic(Y)$, since $H^1(X,\calO_X)=0$ (cf. \cite{H1} Ex. III.12.6). Using the K\"unneth formula we see that the fiber product $X \times Y$ remains a Horrocks variety.  In particular, multiprojective spaces $P = \PP^{n_1} \times ... \times \PP^{n_k}$ are Horrocks schemes if each $n_i \geq 3$.
\end{ex}

\begin{ex}
Weighted projective spaces $W=\PP(a_0,...,a_n)$ with $n \geq 3$ are singular Horrocks schemes, see \cite{D} Section 2.
\end{ex}

\begin{ex}
Any global complete intersection $X \subset \PP^N$ of dimension $n \geq 3$ is necessarily a Horrocks scheme, since the Lefschetz theorem on Picard groups implies that $\Pic(X) \cong \ZZ$ and we know that $H^i(X,\calO_X(m))=0$ for 
$0< i < \dim X$ and every $m \in \ZZ$, see \cite{H2} Chapter IV Section 3 and \cite{H1} III Ex. 5.5(c).
\end{ex}

\begin{ex}
Any Grassmannian $G$ of dimension $n \geq 3$ is a Horrocks scheme.  Since $\Pic(G) \cong \ZZ$ let $\calO(1)$ 
denote the ample generator.  Then $H^i(G,\calO(m))=0$ for $i=1,2$ and $m<0$ by Kodaira vanishing, and 
for $m \geq 0$, $H^i(G,\calO(m)) \cong H^{n-i}(G,\calO(-m) \otimes \omega_G) = 0$ for $i=1,2$ by Serre duality, 
Kodaira vanishing, and the fact that $G$ is Fano.
\end{ex}

\begin{ex}
Let $X$ be a smooth projective Horrocks variety.  Let $E$ be a direct sum of $r \geq 4$ line bundles on $X$, and consider the projectivized space bundle $P:= \PP(E) \overset{\pi} \ra X$, where $\PP(E)=$Proj$(Sym(E))$.  We claim that $P$ is a Horrocks scheme as well.  We already know that $P$ is a smooth projective variety with $\Pic(P) \cong \ZZ \cdot \calO_P(1) \oplus \Pic(X)$.  Thus, any line bundle on $P$ is isomorphic to one of the form $M:=\calO_P(m) \otimes \pi^{*} L$, where $m \in \ZZ$ and $L$ is a line bundle on $X$.  Since the fibers of $\pi$ are all isomorphic to $\PP^{r-1}$ with $r-1 \geq 3$, we have that $$R^i\pi_{*}M=0 \; \; \; \; \mathrm{for} \; \; i=1,2.$$  

For $m<0$ we have $\pi_{*}\calO_P(m)=0$ and for $m \geq 0$ we have $\pi_{*}\calO_P(m)=S^m(E)$, which is each  isomorphic to a direct sum of line bundles on $X$ since $E$ is a direct sum of line bundles on $X$.  Since $X$ is a Horrocks scheme, we have that $$H^i(X,\pi_{*}M) = H^i(X, S^m(E) \otimes L)=0 \; \; \; \; \mathrm{for} \; \; i=1,2.$$  

The Leray spectral sequence implies immediately that $H^1(P,M) = H^1(X,\pi_{*}M)=0$, and it also follows that $H^2(P,M) = E^{0,2}_{\infty} \oplus E^{1,1}_{\infty} \oplus E^{2,0}_{\infty} $, where $E^{p,q}_2 = H^p(X,R^q\pi_{*}M)$ abuts to $H^{p+q}(P,M)$.  But each $E^{p,q}_2$ with $p+q=2$ is $0$ by the above vanishings, hence the infinity pages vanish also.  Thus we see that $P=\PP(E)$ is also a Horrocks scheme.
\end{ex}

\begin{rem}
In the previous example, we do not actually need $E$ to be a direct sum of line bundles provided it has a sufficiently 
short resolution by direct sums of line bundles.  For example, on $\PP^n$ the Euler sequence gives 
$$ 0 \ra \calO \ra \calO(1)^{\oplus (n+1)} \ra T_{\PP^n} \ra 0$$
so we still have that $H^i(\PP^n,S^m(T_{\PP^n})(k))=0$ for $i=1,\ldots,n-2$ and $m,k \in \ZZ$.  Hence, 
$\PP(T_{\PP^n})$ is a splitting scheme for $n \geq 3$ and a Horrocks scheme for $n \geq 4$.
\end{rem}

\begin{prop}
Let $X$ be a smooth projective splitting variety of dimension $n \geq 4$.  Then a vector bundle $E$ on $X$ splits iff 
\begin{enumerate}
\item there exists an ample effective divisor $D$ on $X$ such that $E_{|D}$ splits over $D$, and

\item $H^1(X,E \otimes L)=0$ for every line bundle $L$ on $X$.
\end{enumerate}
\end{prop}

\begin{proof}
To show necessity, suppose that $E \cong \bigoplus L_i$ on $X$, where $L_i$ are line bundles on $X$.  Then for any closed subscheme, $Z \subset X$, $E_{|Z} \cong \bigoplus {L_i}_{|Z}$ splits as well, showing (1).  The assumption that $X$ is a splitting scheme gives (2).

We show sufficiency. By (1) we have an isomorphism $\varphi: \bigoplus M_i \overset{\sim} \ra  E_{|D}$, where $M_i$ are line bundles on 
$D$.  The Grothendieck-Lefschetz theorem tells us that the restriction map gives an isomorphism $\Pic(X) \overset{\sim} \ra \Pic(D)$ of Picard groups.  Lift each $M_i$ on $D$ uniquely (up to isomorphism) to $L_i$ on $X$ and set $F:= \bigoplus L_i$.  Now tensor the short exact sequence

\[ 0 \ra \calO_X(-D) \ra \calO_X \ra \calO_D \ra 0\]

with $F^* \otimes E \cong \calHom(F,E)$ and take cohomology to get an exact sequence

\[ H^0(X,F^* \otimes E) \ra H^0(Y,F^*_{|D} \otimes E_{|D}) \ra H^1(X,F^* \otimes E \otimes \calO_X(-D))=0 \]

where the third vector space vanishes by (2).  Thus we have a surjection 

\[ \hom(F,E) \onto \hom(F_{|D},E_{|D}) \]

so we may lift our isomorphism $\varphi$ to a homomorphism $\psi: F \ra E$, and we claim $\psi$ is an isomorphism.  First, observe that $E$ and $F$ have the same rank and the same first Chern class since $\calO_X(c_1(E)-c_1(F))_{|D} \cong \calO_D$ implies that $\det(F) \cong \calO_X(c_1(F)) \cong \calO_X(c_1(E)) \cong \det(E)$ because of the aforementioned isomorphism on Picard groups. Then $\psi$ induces $\det(\psi): \det(F) \ra \det(E)$ which gives a section 

\[ \det(\psi) \in H^0(X,\det(F)^{-1} \otimes \det(E)) \cong H^0(X,\calO_X(c_1(E)-c_1(F))) \cong H^0(X,\calO_X) \cong \CC \]

which means that $\det(\psi)$ is multiplication by a constant.  But $\det(\psi)$ restricts to an isomorphism $\det(\varphi)$ on $D$, hence must be a non-zero constant and hence invertible, thus showing that $\psi$ is indeed an isomorphism.

\end{proof}

\begin{rem}
From the proof one sees that the sufficiency holds for arbitrary smooth projective varieties of dimension $\geq 4$, but the assumption that $X$ is a splitting scheme gives necessity.
\end{rem}

The following proposition was pointed out to the author by N. Mohan Kumar. 

\begin{prop}
Let $X$ be a smooth projective variety of dimension $\geq 4$.  The following are equivalent:
\begin{enumerate}

\item $X$ is a Horrocks scheme

\item every ample effective divisor $D$ on $X$ is a splitting scheme

\item there exists an ample effective divisor $D$ on $X$ which is a splitting scheme
\end{enumerate}
\end{prop}

\begin{proof}
(1) $\Rightarrow$ (2): tensoring the short exact sequence 
\[ 0 \ra \calO_X(-D) \ra \calO_X \ra \calO_D \ra 0 \]
with a line bundle $L$ on $X$ and taking cohomology we get 
\[ ... \ra H^1(X,L) \ra H^1(D,L_{|D}) \ra H^2(X,L \otimes \calO_X(-D)) \ra ... \]
The outside terms vanish by assumption, and every line bundle on $D$ is isomorphic to one of the form $L_{|D}$, for some line bundle $L$ on $X$ by the Grothendieck-Lefschetz theorem.  Hence $D$ is a splitting scheme by definition. 

(2) $\Rightarrow$ (3): trivial

(3) $\Rightarrow$ (1):
Given $D \subset X$ an ample effective codimension $1$ splitting scheme, consider the short exact sequence 
\[ 0 \ra L \otimes \calO_X((k-1)D) \ra L \otimes \calO_X(kD) \ra L_{|D} \otimes \calO_X(kD)_{|D} \ra 0 \]
for every $k \in \ZZ$.  Since $D$ is a splitting scheme, we have that $H^1(D,L_{|D} \otimes \calO_X(kD)_{|D})=0$ 
for every $k \in \ZZ$, which after taking cohomology gives surjections 
\[ H^1(X,L \otimes \calO_X((k-1)D)) \onto H^1(X,L \otimes \calO_X(kD)) \]
and injections
\[ H^2(X,L \otimes \calO_X((k-1)D)) \into H^2(X,L \otimes \calO_X(kD)) \]
for every $k \in \ZZ$.  Since $D$ is ample, we may take $k \ll 0$ and $k \gg 0$ respectively and use Serre vanishing 
to see that $H^1(X,L)=H^2(X,L)=0$ for any line bundle $L$ on $X$.
\end{proof}

The following corollary is a generalization of Horrocks' criterion for projective $n \geq 3$ space.

\begin{cor}
Let $X$ be a smooth projective Horrocks variety of dimension $n \geq 4$.  A vector bundle $E$ on $X$ splits iff its 
restriction $E_{|D}$ to an ample effective divisor $D \subset X$ splits.
\end{cor}

\begin{proof}
We show the nontrivial direction.  Assuming $E_{|D}$ splits over $D$, we see that condition (1) of Proposition 4.11 is immediately satisfied, so it suffices to check condition (2).  Let $L$ be any line bundle on $X$ and tensor the short exact sequence 
\[ 0 \ra \calO_X(-D) \ra \calO_X \ra \calO_D \ra 0 \]
with $E \otimes L \otimes \calO_X(mD)$, and take cohomology to get
\[ H^1(X,E \otimes L \otimes \calO_X((m-1)D) \ra H^1(X,E \otimes L \otimes \calO_X(mD) \ra H^1(D,E_{|D} \otimes L_{|D} \otimes \calO_X(mD)_{|D}) \]
exact.  By the previous proposition, $D$ is a splitting scheme hence the third term vanishes for any $m \in \ZZ$ since 
$E_{|D}$ splits as a sum of line bundles.  So we have surjections
\[ H^1(X,E \otimes L \otimes \calO_X((m-1)D) \onto H^1(X,E \otimes L \otimes \calO_X(mD)) \]
for every $m \in \ZZ$.  Taking $m \ll 0$ and using Serre duality, we can make the left hand side zero since $D$ is ample, and the surjections above imply that the cohomology must vanish for all integers $m$.  In particular, taking $m=0$ we have that $H^1(X,E \otimes L)=0$.  Since $L$ was arbitrary, we have shown condition (2), which completes the proof.
\end{proof}

There is a natural extension of this result using induction.

\begin{cor}
Let $X$ be a smooth projective variety of dimension $n \geq 4$ such that $H^i(X,L)=0$ for $i=1,\ldots,d-1$ and 
every line bundle $L$, where $d \geq 3$.  Suppose $H_1, \ldots H_{n-d}$ are ample divisors such that 
$H_1 \cap \ldots \cap H_l$ is smooth for each $l=1,\ldots,n-d$.  Set $Y=H_1 \cap \ldots \cap H_{n-d}$, which by 
assumption is a smooth global complete intersection in $X$ of dimension $d \geq 3$.  
Then a vector bundle $E$ on $X$ splits iff its restriction $E_{|Y}$ splits.
\end{cor}

We finish with two examples whose details are easy to check.

\begin{ex}
Consider the quadric surface $S:= \PP^1 \times \PP^1 \into \PP^3 =:X$ inside $\PP^3$ via the Segre embedding.  This is an ample surface with $\calO_X(S)=\calO_X(2)$, $\calO_S(-S)=\calO_S(-1,-1)$, and $\Omega_S \cong \calO_S(-2,0) \oplus \calO_S(0,-2)$.  Taking $E:= \Omega_X$ to be the rank $3$ cotangent bundle on $X=\PP^3$, we have the exact sequence 
\[ 0 \ra \calO_S(-1,-1) \ra {\Omega_X}_{|S} \ra \calO_S(-2,0) \oplus \calO_S(0,-2) \ra 0 \]
The obstruction to this short exact sequence splitting lies in 
\[ \ext^1_{\calO_S}(\calO_S(-2,0) \oplus \calO_S(0,-2) , \calO_S(-1,-1)) \cong H^1(S,\calO_S(1,-1)) \oplus H^1(S,\calO_S(-1,1))=0 \]
by the K\"{u}nneth formula.  Hence 
\[ E_{|S} = {\Omega_X}_{|S} \cong \calO_S(-2,0) \oplus \calO_S(-1,-1) \oplus \calO_S(0,-2) \]
splits on a smooth ample surface $S$ in $X=\PP^3$, but $E=\Omega_{\PP^3}$ itself does not split over $\PP^3$ since $H^1(\PP^3, \Omega_{\PP^3}) \cong \CC \neq 0$.
\end{ex}

\begin{ex}
Take $Y:= \PP^1 \times \PP^2 \subset \PP^2 \times \PP^2 =:X$ defined by the ideal $\calO_X(-1,0)$, and denote by $p_1$ 
and $q_1$ the projection to the first factor of $X$ and $Y$ respectively.  Then $\calO_X(Y)=\calO_X(1,0)$ is nef but not ample.  Letting $E:=p_1^* \Omega_{\PP^2}$, we see that $E$ is a nonsplitting rank $2$ vector bundle on $X$, since by the K\"{u}nneth formula $H^1(X,L)=0$ for any line bundle $L$ on $X$, so if $E$ were to split we must have $H^1(X,E)=0$.  However, by the same formula we see that $H^1(X,E) \cong \CC \neq 0$, so $E$ does not split over $X$.  But, 
\[ E_{|Y} \cong q_1^*({\Omega_{\PP^2}}_{|\PP^1}) \cong \calO_Y(-2,0) \oplus \calO_Y(-1,0) \]
splits over $Y$.
\end{ex}

\end{document}